\documentclass{amsart}

\usepackage{mathrsfs}
\usepackage{amsfonts}
\usepackage{amsmath}
\usepackage{amssymb}
\usepackage{mathtools}
\usepackage{float}
\usepackage{xcolor}
\usepackage[inline]{enumitem}
\usepackage{csquotes}
\usepackage{subcaption}
\usepackage{tikz}
\usepackage{tikz-cd}
\usetikzlibrary{arrows,automata}
\usepackage{hyperref}
\usepackage{centernot}
\usepackage{varwidth}

\usepackage{tgtermes}
\newcommand{\pres}[3]{\textnormal{#1} \langle #2 \mid #3 \rangle}

\newcommand{\lra}[1]{\xleftrightarrow{\ast}_{#1}}
\newcommand{\xra}[1]{\xrightarrow{}^\ast_{#1}}
\newcommand{\xr}[1]{\xrightarrow{}_{#1}}

\newcommand{\trev}{\text{rev}}
\newcommand{\CF}{\mathcal{C}_{\operatorname{cf}}}

\newcommand{\DCF}{\mathcal{C}_{\operatorname{dcf}}}
\newcommand{\REG}{\mathcal{C}_{\operatorname{reg}}}
\newcommand{\IND}{\mathcal{C}_{\operatorname{ind}}}

\newcommand{\cc}{\mathcal{C}}
\newcommand{\ct}{\mathcal{T}}
\newcommand{\cR}{\mathcal{R}}
\newcommand{\cl}{\mathcal{L}}
\newcommand{\cP}{\mathcal{P}}

\newcommand{\loz}{\lozenge}

\DeclareMathOperator{\AFL}{AFL}

\DeclareMathOperator{\IP}{IP}
\DeclareMathOperator{\WP}{WP}

\newtheorem{theorem}{Theorem}[section] 

\newtheorem{lemma}[theorem]{Lemma}     
\newtheorem{corollary}[theorem]{Corollary}

\newtheorem{proposition}[theorem]{Proposition}

\theoremstyle{definition}
\newtheorem{definition}{Definition}

\newtheorem{example}{Example}

\numberwithin{equation}{section}

\newcommand{\aWP}{\WP[\alpha]}
\newcommand{\aaWP}{\WP[\alpha \sqcap \alpha]}

\begin{document}

\title{On the word problem for weakly compressible monoids}

\author{Carl-Fredrik Nyberg-Brodda}
\address{Alan Turing Building, Department of Mathematics, University of Manchester, United Kingdom.}
\email{carl-fredrik.nybergbrodda@manchester.ac.uk}

\thanks{The author gratefully acknowledges funding from the Dame Kathleen Ollerenshaw Trust, which is supporting his current research at the University of Manchester.}


\date{\today}



\begin{abstract}
We study the language-theoretic properties of the word problem, in the sense of Duncan \& Gilman, of weakly compressible monoids, as defined by Adian \& Oganesian. We show that if $\cc$ is a reversal-closed super-$\AFL$, as defined by Greibach, then $M$ has word problem in $\cc$ if and only if its compressed left monoid $L(M)$ has word problem in $\cc$. As a special case, we may take $\cc$ to be the class of context-free or indexed languages. As a corollary, we find many new classes of monoids with decidable rational subset membership problem. Finally, we show that it is decidable whether a one-relation monoid containing a non-trivial idempotent has context-free word problem. This answers a generalisation of a question first asked by Zhang in 1992. 
\end{abstract}

\maketitle



\section{Introduction}

\noindent The word problem for finitely presented monoids was introduced in 1914 by Thue \cite{Thue1914}. Though this problem would eventually turn out to be undecidable in general, via the remarkable work by Markov \cite{Markov1947} and Post \cite{Post1947}, there are many tractable cases. Particularly the word problem for \textit{one-relation} monoids has garnered a great deal of attention. Though it remains a tantalising open problem whether this problem is decidable, much can be said (see especially the recent survey \cite{NybergBrodda2021b}). Thue himself studied this problem, and solved the problem for some cases when the single defining relation is of the form $w = 1$. Such monoids are known as \textit{special} one-relation monoids. Adian \cite{Adian1960b, Adian1966} studied special monoids in-depth, and proved that the word problem is decidable for all special one-relation monoids, via a reduction to the word problem for one-relator groups; this latter problem is decidable by Magnus' famous theorem \cite{Magnus1932}. Adian's student Makanin \cite{Makanin1966, Makanin1966b, NybergBrodda2021c} later extended this to show that the word problem for any special $k$-relation monoid reduces to the word problem for some $k$-relator group. The author has recently proved similar language-theoretic reductions for special monoids \cite{NybergBrodda2020b}.

The work by Adian on special monoids would later be extended by Adian and his student Oganesian in the joint work \cite{Adian1978}. They defined a form of ``compression'', which is called \textit{weak compression} by the author in \cite{NybergBrodda2021b}. A monoid $M$ with defining relations $u_i = v_i$ (for $1 \leq i \leq k$) is called \textit{weakly compressible} if there is some non-empty word $\alpha$ such that every $u_i$ and $v_i$ begins and ends with $\alpha$. Associated to any weakly compressible monoid is a ``compressed'' \textit{left monoid} $L(M)$, in which the sum of the lengths of the defining relations are shorter than in the original monoid. The main idea is that the word problem for $M$ can be reduced to the word problem for $L(M)$. In particular, by combination with Adian's result, this solves the word problem in the class of \textit{subspecial} one-relation monoids, being the class of monoids which can be compressed to a special one-relation monoid. Subspecial monoids were studied already by Lallement \cite{Lallement1974}, and are (essentially) the one-relation monoids containing a non-trivial idempotent. 

Weak compression of subspecial monoids has received some further attention in the past few decades. In Zhang \cite{Zhang1992e} extended the above results on the word problem to solve the conjugacy problem in certain subspecial monoids. Kobayashi \cite{Kobayashi2000} used compression to prove that all one-relation monoids satisfy the topological finiteness condition $\operatorname{FDT}$, for which the subspecial case was the only outstanding case. Building on this, Gray \& Steinberg showed that all one-relation monoids satisfy the homological finiteness property $\operatorname{FP}_\infty$, at the same time studying the algebraic properties of subspecial monoids \cite{Gray2019}.

In another direction, a highly successful approach to studying the word problem for \textit{groups} came through the methods of formal language theory. Such methods were introduced to group theory by An{\={\i}}s{\={\i}}mov \cite{Anisimov1971}, and focus on the language-theoretic properties of the set of all words representing the identity element, calling this set the \textit{word problem} of the group (its properties are generally independent of the finite generating set chosen). An{\={\i}}s{\={\i}}mov showed that the word problem of a group is a regular language if and only if the group is finite, and that the class of groups with context-free word problem is closed under free products. A little over ten years later, Muller \& Schupp (supplemented by \cite{Dunwoody1985}) gave a remarkable algebraic charactersiation, generally referred to as the Muller-Schupp theorem: a finitely generated group has context-free word problem if and only if it is virtually free \cite{Muller1983}. 

The language of words representing the identity element in a monoid is generally not very enlightening (though for special monoids it is \cite{NybergBrodda2020b}, as in the case of groups). To remedy this, Duncan \& Gilman introduced a new language which encodes equality in a monoid \cite{Duncan2004}. This set, which is also termed the \textit{word problem} for the monoid, has many of the attractive properties of the analogous set for groups, such as invariance under generating set chosen. Furthermore, it is easy to show that a monoid has regular word problem if and only if it is finite, mimicking An{\={\i}}s{\={\i}}mov's result. The word problem also generally behaves well under taking free products \cite{Brough2019, NybergBrodda2021f}. This notion of the word problem has also been studied in e.g. \cite{Holt2008, Hoffmann2012}.

This paper seeks to study the language-theoretic properties of Adian \& Oganesian's weak compression. We will paint a fairly complete picture. The main theorem of this article (Theorem~\ref{Thm:MAIN}) will involve classes of languages $\cc$ which are \textit{super-$\AFL$s}, as introduced by Greibach \cite{Greibach1970}. Such classes generalise the context-free languages and the indexed languages; roughly speaking, an $\AFL$ is a super-$\AFL$ if it is closed under ``recursion''  (see \S\ref{Subsec:SuperAFL} for details). We will prove the following main theorem: 
\begin{theorem}\label{Thm:MAIN}
Let $M$ be a weakly compressible monoid, and let $L(M)$ be its left monoid. Let $\cc$ be a reversal-closed super-$\AFL$. Then $M$ has word problem in $\cc$ if and only if $L(M)$ has word problem in $\cc$.
\end{theorem}

An overview of the article is as follows. In \S\ref{Sec:NotationEtc} we give some notation, useful concepts, and define two language-theoretic operations from \cite{NybergBrodda2021f}. In \S\ref{Sec:Weak_compression}, we define weak compression, amalgamating several authors' definitions and notations. In \S\ref{Sec:Proof}, we prove Theorem~\ref{Thm:MAIN}. Finally, in \S\ref{Sec:Context-free} we give some corollaries of Thoerem~\ref{Thm:MAIN}, particularly focussed on context-free monoids. In particular, we will obtain many monoids for which the rational subset membership problem is decidable (Corollary~\ref{Cor:CF_L(M)_has_dec_RSMP}); and that it is decidable whether a one-relation monoid containing a non-trivial idempotent has context-free word problem (Theorem~\ref{Thm:Dec_if_OR_has_CF}). 

This article is a condensed form of Chapter~4 of the author's PhD thesis \cite{Thesis}.

\section{Notation and auxiliary results}\label{Sec:NotationEtc}

\noindent We assume the reader is familiar with the fundamentals of formal language theory. In particular, an $\AFL$ (\textit{abstract family of languages}) is a class of languages closed under homomorphism, inverse homomorphism, intersection with regular languages, union, concatenation, and the Kleene star. For some background on this, and other topics in formal language theory, we refer the reader to standard books on the subject \cite{Harrison1978, Hopcroft1979, Berstel1985}. The paper also assumes familiarity with the basics of the theory of semigroup, monoid, and group presentations, which will be written as $\pres{Sgp}{A}{\cR}$, $\pres{Mon}{A}{\cR}$, and $\pres{Gp}{A}{\cR}$, respectively. For further background see e.g. \cite{Adian1966,Magnus1966, Lyndon1977,Campbell1995}. We also refer the reader to \cite{NybergBrodda2020b, NybergBrodda2021f}.

\subsection{Monoids, words, rewriting}

Let $A$ be a finite alphabet, and let $A^\ast$ denote the free monoid on $A$, with identity element denoted $\varepsilon$ or $1$, depending on the context. Let $A^+$ denote the free semigroup on $A$, i.e. $A^+ = A^\ast - \{ \varepsilon\}$. For $u, v \in A^\ast$, by $u \equiv v$ we mean that $u$ and $v$ are the same word. For $w \in A^\ast$, we let $|w|$ denote the \textit{length} of $w$, i.e. the number of letters in $w$. We have $|\varepsilon| = 0$. If $w \equiv a_1 a_2 \cdots a_n$ for $a_i \in A$, then we let $w^\trev$ denote the \textit{reverse} of $w$, i.e. the word $a_n a_{n-1} \cdots a_1$. For a language $L$, we define $L^\trev$ as the collection of all words $w^\trev$ where $w \in L$. We say that a class of languages $\cc$ is \textit{reversal-closed} if for every $L \in \cc$, we have $L^\trev \in \cc$. We say that the word $w \in A^\ast$ is \textit{self-overlap free} if it is empty, or else if it is non-empty and none of the proper non-empty prefixes of $w$ is also a proper non-empty suffix of $w$. Thus $xyxyy$ is self-overlap free, but $xyxyx$ is not. If the words $u, v \in A^\ast$ are equal in the monoid $M = \pres{Mon}{A}{\cR}$, then we denote this $u =_M v$. Finally, when we say that a monoid $M$ is generated by a finite set $A$, we mean that there exists a surjective homomorphism $\pi \colon A^\ast \to M$. In this case, $u =_M v$ will be used synonymously with $\pi(u) = \pi(v)$. 

We give some notation for rewriting systems. For an in-depth treatment and further explanations of the terminology, see e.g. \cite{Jantzen1988, Book1993}. A \textit{rewriting system} $\cR$ on $A$ is a subset of $A^\ast \times A^\ast$. An element of $\cR$ is called a \textit{rule}. The system $\cR$ induces several relations on $A^\ast$. We will write $u \xr{\cR} v$ if there exist $x, y \in A^\ast$ and a rule $(\ell, r) \in \cR$ such that $u \equiv x\ell y$ and $v \equiv xry$. We let $\xra{\cR}$ denote the reflexive and transitive closure of $\xr{\cR}$. We denote by $\lra{\cR}$ the symmetric, reflexive, and transitive closure of $\xr{\cR}$. The relation $\lra{\cR}$ defines the least congruence on $A^\ast$ containing $\cR$. For $X \subseteq A^\ast$, we let $\langle X \rangle_\cR$ denote the set of \textit{ancestors} of $X$, i.e. $\langle X \rangle_\cR = \{ w \in A^\ast \mid \exists x \in X \textnormal{ such that } w \xra{\cR} x \}$. The monoid $\pres{Mon}{A}{\cR}$ is identified with the quotient $A^\ast / \lra{\cR}$. For a rewriting system $\ct \subseteq A^\ast \times A^\ast$ and a monoid $M = \pres{Mon}{A}{\cR}$, we say that $\ct$ is $M$\textit{-equivariant} if for every rule $(u, v) \in \ct$, we have $u =_M v$. That is, $\ct$ is $M$-equivariant if and only if $\lra{\ct} \subseteq \lra{\cR}$.

A rewriting system $\cR \subseteq A^\ast \times A^\ast$ is said to be \textit{monadic} if $(u, v) \in \cR$ implies $|u| \geq |v|$ and $v \in A \cup \{ \varepsilon \}$. We say that $\cR$ is \textit{special} if $(u, v) \in \cR$ implies $v \equiv \varepsilon$. Every special system is monadic. Let $\cc$ be a class of languages. A monadic rewriting system $\cR$ is said to be $\cc$ if for every $a \in A \cup \{ \varepsilon \}$, the language $\{ u \mid (u, a) \in \cR \}$ is in $\cc$. Thus, we may speak of e.g. monadic $\cc$-rewriting systems or monadic context-free rewriting systems.

Let $M$ be a monoid with a finite generating set $A$. The (monoid) \textit{word problem of $M$ with respect to $A$} is defined as the language
\begin{equation}\label{Eq:WP_A^M_def}
\WP_A^M := \{ u \# v^\trev \mid u, v \in A^\ast, u =_M v\},
\end{equation}
where $\#$ is some fixed symbol not in $A$. For a class of languages $\cc$, we say that $M$ has $\cc$-word problem if $\WP_A^M$ is in $\cc$. If $\cc$ is closed under inverse homomorphism, then $M$ having $\cc$-word problem does not depend on the finite generating set $A$ chosen for $M$ (see e.g. \cite{Hoffmann2012}). The above definition \eqref{Eq:WP_A^M_def} is due to Duncan \& Gilman \cite[p.\ 522]{Duncan2004}. We will throughout this article use $A_\#$ as a short-hand for the set $A \cup \{ \# \}$. 

\subsection{Super-$\AFL$s}\label{Subsec:SuperAFL}

The theorems in this paper involve a special type of classes of languages, called \textit{super-$\AFL$s}. These were introduced by Greibach \cite{Greibach1970}, using \textit{nested iterated substitutions}. In \cite{NybergBrodda2021f}, the author gave an equivalent characterisation of super-$\AFL$s, which we shall use here. We begin with a simple definition. Recall that for a language $L$ and a rewriting system $\cR$, the language $\langle L \rangle_\cR$ denotes the language of ancestors of $L$ under $\cR$, i.e. the set of words which can be rewritten to some word in $L$. 

\begin{definition}
Let $\cc$ be a class of languages. Let $\cR \subseteq A^\ast \times A^\ast$ be a rewriting system. Then we say that $\cR$ is $\cc$\textit{-ancestry preserving} if for every $L \subseteq A^\ast$ with $L \in \cc$, we have $\langle L \rangle_\cR \in \cc$. If every monadic $\cc$-rewriting system is $\cc$-ancestry preserving, then we say that $\cc$ has the \textit{monadic ancestor property}.
\end{definition}

\begin{example}
If $\cR \subseteq A^\ast \times A^\ast$ is a monadic context-free rewriting system, and $L \subseteq A^\ast$ is a context-free language, then $\langle L \rangle_\cR$ is a context-free language \cite[Theorem~2.2]{Book1982b}. That is, if we let $\CF$ denote the class of context-free languages, then every monadic $\CF$-rewriting system is $\CF$-ancestry preserving; in other words, the class of context-free languages has the monadic ancestor property. 
\end{example}

\begin{definition}\label{Def:super-AFL}
Let $\cc$ be an $\AFL$. Then $\cc$ is said to be a \textit{super-$\AFL$} if $\cc$ has the monadic ancestor property.
\end{definition}

For example, the class $\CF$ of context-free languages forms a super-$\AFL$ by \cite[Theorem~1.2]{Kral1970}, and the class $\IND$ of indexed languages is also a super-$\AFL$ \cite{Aho1968, Engelfriet1985}. Both are closed under reversal. On the other hand, neither the classes $\REG$ nor $\DCF$ of regular resp. deterministic context-free languages are super-$\AFL$s. Indeed, if $\cc$ is a super-$\AFL$, then $\CF \subseteq \cc$, by \cite[Theorem~2.2]{Greibach1970}. For more examples and generalisations, we refer the reader to the so-called \textit{hyper-$\AFL$s} defined by Engelfriet \cite{Engelfriet1985}, all of which are super-$\AFL$s.

We remark that Greibach \cite{Greibach1970} originally defines super-$\AFL$s via ``nested iterated substitutions'', which behave similarly to taking ancestors under monadic rewriting systems. The author \cite[Proposition~2.2]{NybergBrodda2021f} has proved that the above definition, using rewriting systems, is equivalent to Greibach's, and so we will use her results about super-$\AFL$s without restriction. 

\subsection{Alternating products and bipartisan ancestry}
With general definitions taken care of, we will now turn to give some useful auxiliary language-theoretic results. 

We first define an operation (the alternating product) on certain languages, which mimics the operation of the free product of semigroups. This operation appears in \cite{NybergBrodda2021f}. Fix an alphabet $A$ and let $\#$ be a symbol not in $A$. Let $L \subseteq A^\ast \# A^\ast$ be any language. We say that $L$ is \textit{concatenation-closed} (with respect to $\#$) if 
\begin{equation}
u_1 \# v_1 \in L \textnormal{ and } u_2 \# v_2 \in L \implies u_1 u_2 \# v_2 v_1 \in L,
\end{equation}
where $u_1, v_1, u_2, v_2 \in A^\ast$. The word problem of any finitely generated monoid is always a concatenation-closed language. 

Let $X \colon \mathbb{N} \to \{ 1, 2 \}$ be a parametrisation such that either $X(2j) = 1$ and $X(2j+1) = 2$, or else $X(2j) = 2$ and $X(2j+1) = 1$, for all $ j \in \mathbb{N}$. A parametrisation $X$ of this form will be called \textit{standard}. Given two concatenation-closed languages in $A^\ast \# A^\ast$, we now present an operation for combining them. Let $L_1, L_2 \subseteq A^\ast \# A^\ast$ be concatenation-closed. Then the \textit{alternating product} $L_1 \star L_2$ of $L_1$ and $L_2$ is defined as the language consisting of all words of the form: 
\begin{equation}\label{Eq:Def_alternating}
u_1 u_2 \cdots u_k \# v_k \cdots v_2 v_1,
\end{equation}
where for all $1 \leq i \leq k$ we have $u_i \# v_i \in L_{X(i)}$, where $X$ is a standard parametrisation. In particular, $\# \in L_1 \star L_2$ if and only if $\# \in L_1$ or $\# \in L_2$. Alternating products are modelled on the (semigroup) free product, as the following example shows. 

\begin{example}
Let $L_1 = \{ a^n \# a^n \mid n \geq 0 \}$ and $L_2 = \{ b^n \# b^n \mid n \geq 0 \}$. Then, of course, $L_1 = \WP_{\{ a \}}^{\{ a \}^\ast}$ and $L_2 = \WP_{\{ b \}}^{\{ b \}^\ast}$. Now both languages $L_1$ and $L_2$ are concatenation-closed, and it is easy to see that we have
\[
L_1 \star L_2 = \{ a^{n_1} b^{n_2} \cdots  a^{n_k} \# a^{n_k} \cdots b^{n_2} a^{n_1} \mid k \geq 0; n_1, n_2, \dots, n_{k-1} \geq 1; n_k \geq 0 \}.
\]
Thus $L_1 \star L_2 = \{ w \# w^\trev \mid w \in \{ a, b \}^\ast \} = \WP_{\{a,b\}}^{\{a,b\}^\ast}$.
\end{example}

Using the monadic ancestor property, one can prove the following useful statement. 

\begin{proposition}[{\cite[Proposition~2.3]{NybergBrodda2021f}}]\label{Prop:Alternating_preserves}
Let $\cc$ be a super-$\AFL$. Let $L_1, L_2 \subseteq A^\ast \# A^\ast$ be concatenation-closed languages. Then $L_1, L_2 \in \cc \implies L_1 \star L_2 \in \cc$. 
\end{proposition}

We will require another operation with useful preservation properties, also introduced in \cite{NybergBrodda2021f}. Let $\cR_1, \cR_2 \subseteq A^\ast \times A^\ast$ be two rewriting systems. Let $L \subseteq A^\ast \# A^\ast$. Then we define the \textit{bipartisan $(\cR_1, \cR_2)$-ancestor} of $L$ as the language 
\begin{equation}\label{Eq:Bipartisan_ancestor_def}
L^{\cR_1, \cR_2} = \{ w_1 \# w_2 \mid \exists u_1 \# u_2 \in L \textnormal{ such that } w_i \in \langle u_i \rangle_{\cR_i} \textnormal{ for $i=1, 2$}\}.
\end{equation}

Bipartisan ancestors can, informally speaking, manipulate the left and the right side (of $\#$ in words in $L$) using $\cR_1$ resp. $\cR_2$, and these manipulations can be independent of one another. We give an example of this independence.

\begin{example}
Let $A = \{ a, b\}$, and let $L = \{ a \# a \}$. Let $\cR_1$ be the rewriting system with the rules $(b^n, a)$ for all $n \geq 1$. Then 
\[
L^{\cR_1, \cR_1} = \{ b^{n_1} \# b^{n_2} \mid n_1, n_2 \geq 1\} \cup \{ a\#a \}.
\]
In particular, we have $b^{n_1} \# b^{n_2} \in L^{\cR_1, \cR_1}$ even if $n_1 \neq n_2$.
\end{example}

\begin{proposition}[{\cite[Proposition~2.6]{NybergBrodda2021f}}]\label{Prop:Ancestor_preserves}
Let $\cc$ be a super-$\AFL$. Let $L \subseteq A^\ast \# A^\ast$, and let $\cR_1, \cR_2$ be monadic $\cc$-rewriting systems. Then $L \in \cc \implies L^{\cR_1, \cR_2} \in \cc$.
\end{proposition}

Bipartisan ancestors are useful in describing the language theory of monoid (and group) free products. See \cite{NybergBrodda2021f} for further details. Alternating products and bipartisan ancestors are the language-theoretic tools we shall require for this article. We now turn to the main subject of the article. 

\section{Weak compression}\label{Sec:Weak_compression}

We present weak compression as it is defined in the survey \cite[\S3.1]{NybergBrodda2021b} by the author. This is an amalgamation of other approaches \cite{Lallement1974, Adian1978, Zhang1992e, Kobayashi2000, Gray2018}. 

Let $A$ be an alphabet. We say that a pair $(u, v)$ of words is \textit{sealed} by $w \in A^+$ if $u, v \in wA^\ast \cap A^\ast w$. If a pair is sealed by some word, then it is clearly sealed by a unique self-overlap free word $\alpha$. For example, $(xyxpxyx, xyxqxyx)$ is sealed by $xyx$, but also by the self-overlap free word $x$.

Let $M$ be the monoid defined by the presentation
\begin{equation}\label{Eq:Main_Monoid}
\pres{Mon}{A}{u_i = v_i \: (1 \leq i \leq p)}.
\end{equation}

\begin{definition}
We say that the monoid defined by \eqref{Eq:Main_Monoid} is \textit{weakly compressible} (with respect to $\alpha$) if there is some self-overlap free word $\alpha \in A^+$ such that for all $1 \leq i \leq p$, the pair $(u_i, v_i)$ is sealed by $\alpha$.
\end{definition}

If $M$ is not weakly compressible, then we say that $M$ is \textit{incompressible}. Any special monoid is incompressible by default (in particular groups are incompressible). 

Let $M$ be a weakly compressible monoid defined by \eqref{Eq:Main_Monoid}. We will assume $|A|>1$, for otherwise $M$ is a finite cyclic monoid, and all is trivial. A word $w \in A^\ast$ is called a \textit{left $\alpha$-conjugator} if $w \in \alpha A^\ast$. Let $\Sigma_\ast(\alpha)$ be the set of left $\alpha$-conjugators which contain exactly one occurrence of $\alpha$, i.e. $\Sigma_\ast(\alpha) = \alpha(A^\ast - A^\ast\alpha A^\ast)$. For ease of notation, we write $\Sigma_\ast$ instead of $\Sigma_\ast(\alpha)$. Clearly, $\Sigma_\ast$ is a countably infinite suffix code (as $|A|>1$), and $\Sigma_\ast^+$ is the set of all left $\alpha$-conjugators. Enumerate the words of $\Sigma_\ast$ as $\{ w_1, w_2, \dots \}$ and fix a set $\Gamma_\ast(\alpha) = \Gamma_\ast$ of symbols $\{ \gamma_{w_1}, \gamma_{w_2}, \dots \}$ in bijective correspondence with $\Sigma_\ast$ via the map $\varphi \colon w_i \mapsto \gamma_{w_i}$. As $\Sigma_\ast$ is a suffix code, we can extend $\varphi$ to an isomorphism of the free monoids $\Sigma_\ast^\ast$ and $\Gamma_\ast^\ast$. 

Every defining relation $u_i = v_i$ in \eqref{Eq:Main_Monoid} is sealed by $\alpha$, so $u_i, v_i \in \Sigma_\ast^\ast \alpha$. We factor $u_i, v_i$ uniquely over the suffix code $\Sigma_\ast$, yielding 
\begin{align}\label{Eq:Left_pieces}
u_i \equiv u_{i,1} u_{i,2} \cdots u_{i,m_i} \alpha, \quad \textnormal{and} \quad v_i \equiv v_{i,1} v_{i,2} \cdots v_{i,n_i} \alpha.
\end{align}
Any word $u_{i,j}$ or $v_{i,k}$ appearing in the factorisations \eqref{Eq:Left_pieces} for some $1 \leq i \leq p$ is called a \textit{left piece} of $M$. The set of all left pieces of $M$ is denoted $\Sigma(\alpha)$, which will be shortened to $\Sigma$. As $M$ is finitely presented, $\Sigma$ is a finite set of words. We let $\Gamma = \Gamma(\alpha)$ be the set $\varphi(\Sigma)$ of symbols from $\Gamma_\ast$. 

From the factorisation \eqref{Eq:Left_pieces}, we define a new presentation 
\begin{equation}\label{Eqref:L(M)_pres}
\pres{Mon}{\Gamma_\ast}{\varphi(u_{i,1} u_{i,2} \cdots u_{i,m_i}) = \varphi(v_{i,1} v_{i,2} \cdots v_{i,n_i}) \:\: (1 \leq i \leq p)}.
\end{equation}

\begin{definition}
The monoid defined by the presentation \eqref{Eqref:L(M)_pres} is called the \textit{extended left monoid} associated to the monoid $M$ (defined by \eqref{Eq:Main_Monoid}), and is denoted $L_\ast(M)$. The submonoid of $L_\ast(M)$ generated by $\Gamma$ is the \textit{left monoid} associated to $M$, and is denoted $L(M)$. 
\end{definition}

It is obvious from \eqref{Eqref:L(M)_pres} that $L_\ast(M) = \mathcal{F} \ast L(M)$, where $\ast$ denotes the monoid free product, and $\mathcal{F}$ is a free monoid of countably infinite rank, freely generated by $\Gamma_\ast - \Gamma$. We emphasise that it is always decidable to compute the presentation \eqref{Eqref:L(M)_pres} starting from the weakly compressible \eqref{Eq:Main_Monoid}, as it only requires computing with regular languages. Of course, $L(M)$ is finitely presented, with the same defining relations as in \eqref{Eqref:L(M)_pres}. Furthermore, the sum of the lengths of the defining relations in $L(M)$ is strictly shorter than the corresponding sum for $M$. In particular, setting $L^1(M) = L(M)$ and $L^i(M) = L(L^{i-1}(M))$ for $i > 1$, there exists some $n \geq 1$ such that $L^n(M)$ is incompressible.\footnote{The compression defined by Lallement \cite{Lallement1974} transforms $M$ into $L^n(M)$ in a single step, whereas that by Adian \& Oganesian \cite{Adian1978} or indeed Kobayashi \cite{Kobayashi2000} corresponds to transforming $M$ into $L(M)$.}

We give an example. One can easily verify that if 
\begin{equation}\label{Eq:M1}
M_1 = \pres{Mon}{x,y}{xyyxxxyxxyyxxxy = xy},
\end{equation}
then the defining relation of $M_1$ is sealed by $xy$. Factorising both sides of the defining relation, we find $\Sigma = \{ xyx, xyyxx\}$ and hence $\Gamma = \{ \gamma_{xyx}, \gamma_{xyyxx} \}$. Thus 
\begin{equation}\label{Eq:L(M1)}
L(M_1) = \pres{Mon}{\gamma_{xyx}, \gamma_{xyyxx}}{\gamma_{xyyxx}\gamma_{xyx}\gamma_{xyyxx} = 1}.
\end{equation}
This (special) monoid is incompressible. It is isomorphic to $\pres{Mon}{a,b}{aba=1}$, which is isomorphic to the infinite cyclic group $\mathbb{Z}$ by removing the redundant generator $b$.

We return to the general case, and present the normal form results from \cite{Lallement1974, Adian1978}. First, note that it is obvious that if two words $u, v \in A^\ast$ are equal in $M$, then $u$ contains an occurrence of the self-overlap free word $\alpha$ if and only if $v$ does; and furthermore, if neither $u$ nor $v$ contain $\alpha$, then $u =_M v$ if and only if $u \equiv v$. Second, given any $u \in A^\ast \alpha A^\ast$, there exist unique $u', u'' \in A^\ast - A^\ast \alpha A^\ast$ and $u^\dag \in \alpha A^\ast \cap A^\ast \alpha$ such that $u \equiv u' u^\dag u''$. We call such a factorisation the \textit{canonical form} of $u$, and $u^\dag$ is called the $\alpha$\textit{-part} of $u$. Write $u^\dag \equiv u_\ell^\dag \alpha$, where $u_\ell^\dag \in \Sigma^\ast$. The subword $\phi(u_\ell^\dag)$ of $\phi(u)$ is called the $\gamma$\textit{-part} of $u$. The following theorem is fundamental to weakly compressible monoids. 

\begin{theorem}[Adian \& Oganesian, Lallement]\label{Thm:NFT}
Let $M$ be a weakly compressible monoid defined by \eqref{Eq:Main_Monoid}. Let $u, v \in A^\ast \alpha A^\ast$ have canonical forms $u' u^\dag u''$ and $v' v^\dag v''$, respectively. Let $\gamma_u, \gamma_v$ be the $\gamma$-parts of $u$ and $v$, respectively. Then $u =_M v$ if and only if (1) $u' \equiv v'$ and $u'' \equiv v''$; and (2) $\gamma_u = \gamma_v$ in $L_\ast(M)$. Furthermore, (2) is equivalent to (2') $u^\dag =_M v^\dag$. 
\end{theorem}

In particular, the word problem for $M$ is decidable if and only if the word problem for $L(M)$ is decidable. Furthermore, one can without much difficulty show that the left (right) divisibility problem for $M$ reduces to the word and left (right) divisibility problems in $L(M)$. In particular we can solve the word and divisibility problems in the monoid $M_1$ defined by \eqref{Eq:M1}. We refer the reader to \cite[Theorem~3]{Adian1978} for further details.

\section{Proof of Theorem~\ref{Thm:MAIN}}\label{Sec:Proof}

Let $M$ be a weakly compressible (with respect to some self-overlap free word $\alpha$) monoid defined by  \eqref{Eq:Main_Monoid}, and let $\varphi, \Sigma_\ast, \Gamma_\ast, \Sigma$, and $\Gamma$ be as in \S\ref{Sec:Weak_compression}. We will now prove that the language-theoretic properties of the monoid $M$ defined by \eqref{Eq:Main_Monoid} can be reduced to the properties of the left monoid $L(M)$. Let $\cc$ be a reversal-closed super-$\AFL$, which will remain fixed throughout this section. 

 To simplify some technical notation we shall sometimes consider $\Sigma$, rather than $\Gamma$, as a finite generating set for the monoid $L(M)$, as there exists a surjective homomorphism $\pi_\Sigma$ from $\Sigma^\ast$ onto $L(M)$ given by the composition $\pi_\Sigma = \pi_\Gamma \circ \varphi$, where $\pi_\Gamma \colon \Gamma^\ast \to L(M)$ is the natural surjective homomorphism. However, it is important to notice that if $u, v \in \Sigma^\ast$, then generally $\pi_\Sigma(u) = \pi_\Sigma(v)$ is very distinct from $u =_M v$, unlike what Theorem~\ref{Thm:NFT} may seem to suggest. Instead, using the canonical forms we find that we have 
\begin{equation}\label{Eq:Using_piSigma}
\pi_\Sigma(u) = \pi_\Sigma(v) \quad \iff \quad u \alpha =_M v\alpha.
\end{equation}

By the first remark preceding Theorem~\ref{Thm:NFT}, it follows that the language $\WP_A^M$ is a union $\aWP_A^M \cup \operatorname{W}_\alpha^-$ of the two disjoint languages 
\begin{align}
\aWP_A^M &= \{ w_1 \# w_2^\trev \mid w_1, w_2 \in A^\ast \alpha A^\ast \textnormal{ such that } w_1 =_M w_2 \}, \label{Def:awP} \\
\operatorname{W}_\alpha^- &= \{ w \# w^\trev \mid w \in A^\ast - A^\ast \alpha A^\ast \}.\label{Def:Walph}
\end{align}

The language $\operatorname{W}_\alpha^-$ defined by \eqref{Def:Walph} is the intersection of the context-free language $\WP_A^{A^\ast}$ with the regular language $(A^\ast - A^\ast \alpha A^\ast) \# (A^\ast - A^\ast \alpha A^\ast)^\trev$. In particular, $\operatorname{W}_\alpha^-$ is a context-free language. Any super-$\AFL$ contains the class of context-free languages \cite[Theorem~2.2]{Greibach1970}. It follows that as $\cc$ is a super-$\AFL$, we have that $\aWP_A^M \in \cc$ implies $\WP_A^M \in \cc$, as $\cc$ is closed under union. 

Having reduced the language-theoretic properties of $\WP_A^M$ to those of $\aWP_A^M$, we perform another reduction, which will yield Lemma~\ref{Lem:aaWP=>WP}. Let 
\begin{equation}\label{Eq:aaWP}
\aaWP_A^M = \{ w_1 \# w_2^\trev \mid w_1, w_2 \in \alpha A^\ast \cap A^\ast \alpha \textnormal{ such that } w_1 =_M w_2 \}.
\end{equation}
Of course, we have the strict inclusions 
\[
\aaWP_A^M \subset \aWP_A^M \subset \WP_A^M.
\]
Using this new language, we define the rewriting system 
\begin{equation}\label{Eq:R_alpha}
\cR_\alpha = \bigg\{ (w_1 \# w_2^\trev \to \#) \:\bigg|\: w_1 \# w_2^\trev \in \aaWP_A^M \cup \WP_A^{A^\ast} \bigg\}.
\end{equation}

This is a monadic rewriting system, and as $\cc$ is a super-$\AFL$ and $\aaWP_A^M \in \cc$, it follows that $\cR_\alpha$ is a $\cc$-rewriting system.

\begin{lemma}\label{Lem:aaWP=>WP}
If $\aaWP_A^M \in \cc$, then $\WP_A^M \in \cc$. 
\end{lemma}
\begin{proof}
Suppose $\aaWP_A^M \in \cc$. By the remarks following \eqref{Def:awP} and \eqref{Def:Walph}, it suffices to show that $\aWP_A^M \in \cc$. We claim that 
\begin{equation}\label{Eq:Equality_in_aaWP=>WP}
\aWP_A^M = \langle \WP_A^{A^\ast} \rangle_{\cR_\alpha} \cap (A^\ast \alpha A^\ast\# A^\ast \alpha^\trev A^\ast).
\end{equation}
This suffices to establish that $\aWP_A^M \in \cc$, as $\cc$ is closed under intersection with regular languages; $\WP_A^{A^\ast}$ is a context-free language and hence in $\cc$; and $\cc$ has the monadic ancestor property.

Note first that for every rule $(s, t) \in \cR_\alpha$ we have $s, t \in \WP_A^M$. Hence, if $w \in \langle \WP_A^{A^\ast} \rangle_{\cR_\alpha}$, then it is clear by induction on the number of $\cR_\alpha$-rewritings necessary to transform $w$ into an element of $\WP_A^{A^\ast}$ that $w \in \WP_A^M$. Hence the inclusion $\supseteq$ in \eqref{Eq:Equality_in_aaWP=>WP} is proved. 

Second, if $w \in \aWP_A^M$, then $w \equiv u \# v^\trev$ for some $u, v \in A^\ast \alpha A^\ast$ with $u =_M v$. Let $u' u^\dag u''$ and $v' v^\dag v''$ be the canonical forms of $u$ and $v$, respectively. By Theorem~\ref{Thm:NFT}, we have $u' \equiv v'$, $u'' \equiv v''$, and $u^\dag =_M v^\dag$. Hence 
\begin{equation*}
u' \# (v')^\trev, u'' \# (v'')^\trev \in \WP_A^{A^\ast}.
\end{equation*}
Furthermore, $u^\dag \# (v^\dag)^\trev \in \aaWP_A^M$. Thus from \eqref{Eq:R_alpha} we find
\[
(s, \#) \in \cR_\alpha \quad \textnormal{for all} \quad s \in \{ u' \# (v')^\trev, u'' \# (v'')^\trev, u^\dag \# (v^\dag)^\trev\}.
\]
It follows that 
\[
w \equiv u \# v^\trev \equiv u' u^\dag u'' \# (v'')^\trev (v^\dag)^\trev (v')^\trev \xra{\cR_\alpha} \#.
\]
Thus $w \in \langle \# \rangle_{\cR_\alpha} \subseteq \langle \WP_A^{A^\ast} \rangle_{\cR_\alpha}$. As $w$ was arbitrary, and as both $u$ and $v$ contain $\alpha$, we have proved the inclusion $\subseteq$ in \eqref{Eq:Equality_in_aaWP=>WP}. This establishes the desired equality \eqref{Eq:Equality_in_aaWP=>WP}.
\end{proof}

Thus, to prove (the hard direction of) the main theorem, it suffices to show that the properties of $\aaWP_A^M$ reduce to the properties of $L(M)$. This is non-trivial; a reduction to the properties of $L_\ast(M)$ is suggested by Theorem~\ref{Thm:NFT}, but $L_\ast(M)$ is not a finitely generated monoid, being a free product of $L(M)$ by an infinite rank free monoid $\mathcal{F}$. However, the word problem of $\mathcal{F}$ is, informally speaking, essentially the context-free language $\mathcal{L}_\mathcal{F} = \WP_A^{A^\ast} \cap (\Sigma_\ast - \Sigma) \# (\Sigma_\ast - \Sigma)^\trev$. We will show that $\aaWP_A^M$ is, up to technical details, describable as an alternating product of the word problem of $L(M)$ by this language $\mathcal{L}_\mathcal{F}$, together with an appropriate amount of ancestry.\footnote{The alternating products and ancestry defined and used by the author in \cite{NybergBrodda2021f} were, in fact, first developed by the author to deal with precisely the problem of describing the word problem of $L_\ast(M)$.} By the preservation results for alternating products and ancestry, this yields the proof strategy for reducing $\aaWP_A^M$ to $L(M)$. 

\begin{lemma}\label{Lem:If_L(M)_C_then_aaWP_C}
If $L(M)$ has word problem in $\cc$, then $\aaWP_A^M \in \cc$. 
\end{lemma}

Before giving the proof of this key lemma, we need some setup and auxiliary lemmas. We introduce some notation for convenience. Let $\Sigma_1 = \Sigma$, and let $\Sigma_2 = \Sigma_\ast - \Sigma$. Obviously $\Sigma_1 \cap \Sigma_2 = \varnothing$. Let $\Gamma_i = \varphi(\Sigma_i)$ for $i=1, 2$. As $\Sigma_i$ is a suffix code, it follows that $\varphi(\Sigma_i^\ast) = \varphi(\Sigma_i)^\ast = \Gamma_i^\ast$, for $i=1,2$. For further ease of notation, we write $M_1 = \langle \Gamma_1 \rangle_{L_\ast(M)} = L(M)$, and $M_2 = \langle \Gamma_2 \rangle_{L_\ast(M)} = \mathcal{F}$. Then $L_\ast(M) = M_1 \ast M_2$, where $\ast$ denotes the monoid free product. In this new notation, it follows directly from Theorem~\ref{Thm:NFT} that if $u, v \in \Sigma_1^\ast$, then 
\begin{equation}\label{Eq:SIGMA1ua=va_iff_phi(u)=phi(v)}
u\alpha =_M v\alpha \quad \iff \quad \varphi(u) =_{M_1} \varphi(v),
\end{equation}
as in this case $\varphi(u), \varphi(v) \in \Gamma_1^\ast$, and all defining relations of the presentation \eqref{Eqref:L(M)_pres} of $L_\ast(M)$ are pairs of words over the alphabet $\Gamma_1$. Analogously, if $u, v \in \Sigma_2^\ast$, then
\begin{equation}\label{Eq:SIGMA2ua=va_iff_u=v}
u\alpha =_M v\alpha \quad \iff \quad u \equiv v.
\end{equation}

We define two rewriting systems on $\Sigma_1^\ast$ resp. $(\Sigma_1^\trev)^\ast$. Let 
\begin{align}
I_\alpha &= \{ (w\alpha, \alpha) \mid w \in \Sigma_1^+ \colon w\alpha =_M \alpha \}, \label{Def:Ialpha}\\
I_\alpha^r &= \{ ( (w\alpha)^\trev, \alpha^\trev) \mid w \in \Sigma_1^+ \colon w\alpha =_M \alpha \}.\label{Def:IalphaREV}
\end{align}

Let now $u \in (\Sigma_1 \cup \Sigma_2)^\ast$ be any word, factorised uniquely as $u_0 u_1 \cdots u_n$ with $u_i \in \Sigma_{X(i)}^+$ for all $0 \leq i < n$ and $u_i \in \Sigma_{X(n)}^\ast$, where $X$ is a standard parametrisation. We say that (this factorisation) of $u$ is \textit{reduced} if $u \equiv \alpha$ or $u \equiv \varepsilon$; or if $u_i \alpha \neq_M \alpha$ for all $0 \leq i \leq n$. Obviously, any irreducible descendant of $u$ modulo $I_\alpha$ is reduced, and any reduced word is clearly irreducible modulo $I_\alpha$. Furthermore, as $I_\alpha$ is $M$-equivariant, if $u \xra{I_\alpha} u'$, then $u =_M u'$, so we conclude that every word $u \in (\Sigma_1 \cup \Sigma_2)^\ast$ is equal to some reduced word $u'$ (though this is generally not unique). Given any reduced $u'$, we can uniquely factorise it as a reduced factorisation $u_0'  u_1' \cdots u_k'$, where $u_i' \in \Sigma_{Y(i)}^\ast$ for $0 \leq i \leq k$, where $Y$ is a standard parametrisation. This factorisation $u_0'  u_1' \cdots u_k'$ of $u'$ is called the \textit{normal form} of the reduced word $u'$. 

\begin{lemma}\label{Lem:LastM_is_FP}
Let $u, v \in (\Sigma_1 \cup \Sigma_2)^\ast$. Let $u'$ resp. $v'$ be any reduced forms of $u$ resp. $v$, with normal forms $u' \equiv u_0' u_1' \cdots u_m'$ and $v' \equiv v_0' v_1' \cdots v_n'$ respectively. Then $u\alpha =_M v\alpha$ if and only if (1) $n=m$, and (2) $u_i', v_i' \in \Sigma^\ast_{X(i)}$ and $u_i' \alpha =_M v_i'\alpha$ for all $0 \leq i \leq n$, for some standard parametrisation $X$.
\end{lemma}
\begin{proof}
The ``if'' direction follows by induction on $n$, and the simple observation that if $u_i, v_i, u_{i+1}, v_{i+1} \in \Sigma^+$, then as $u_{i+1}$ and $v_{i+1}$ begin with $\alpha$ we have
\[
(u_i \alpha =_M v_i \alpha \:\:\textnormal{ and } \:\: u_{i+1} \alpha =_M v_{i+1} \alpha) \quad \implies \quad u_i u_{i+1} \alpha =_M v_i v_{i+1} \alpha.
\]

For the ``only if'' direction, by Theorem~\ref{Thm:NFT}, it follows that $u\alpha =_M v\alpha$ if and only if $\varphi(u') =_{M_1 \ast M_2} \varphi(v')$. Now $\varphi(u') \equiv \varphi(u_0') \varphi(u_1') \cdots \varphi(u'_m)$. As $\varphi(u_i') =_{M_1 \ast M_2} 1$ if and only $u_i' \alpha =_M \alpha$, it follows from the fact that $u'$ is reduced that $\varphi(u')$ is reduced with respect to the monoid free product $M_1 \ast M_2$, i.e. no non-empty subword of $\varphi(u')$ equals $1$ in $M_1 \ast M_2$. The analogous statement is true for $\varphi(v')$. Hence, as $\varphi(u') =_{M_1 \ast M_2} \varphi(v')$, it follows from the usual normal form lemma for monoid free products (see \cite[\S8.2]{Howie1995} or \cite[\S1]{NybergBrodda2021f}) that (1) $n=m$; and that (2) $\varphi(u_i'), \varphi(v_i') \in \Gamma_{X(i)}^\ast$ and $\varphi(u_i') =_{M_{X(i)}} \varphi(v_i')$ for $0 \leq i \leq n$, where $X$ is some standard parametrisation. As $u_i' \alpha =_M v_i' \alpha$ if and only $\varphi(u_i') =_{M_{X(i)}} \varphi(v_i')$, the result follows. 
\end{proof}

The rewriting systems $I_\alpha$ and $I_\alpha^r$ defined in \eqref{Def:Ialpha} and \eqref{Def:IalphaREV} are close to being monadic. We show that, via appropriate rational transductions, we can extend the monadic ancestor property to also apply to these rewriting systems, using the self-overlap free property of $\alpha$ in a non-trivial way. 

\begin{lemma}\label{Lem:I_a_ancestrypreserving}
If $L(M)$ has word problem in $\cc$, then $I_\alpha$ and $I_\alpha^r$ are $\cc$-ancestry preserving. 
\end{lemma}
\begin{proof}
Let $L \subseteq (\Sigma_1 \cup \Sigma_2)^\ast$ be an arbitrary language with $L \in \cc$. To prove that $I_\alpha$ is $\cc$-ancestry preserving, it suffices to show that $\langle L \rangle_{I_\alpha} \in \cc$. Let $I_\alpha^+$ be the set of left-hand sides of rules in $I_\alpha$. We claim $I_\alpha^+ \in \cc$. Indeed, $w \in \Sigma_1^\ast$ is such that $w\alpha =_M \alpha$ if and only if $\varphi(w) =_{L(M)} 1$ by \eqref{Eq:SIGMA1ua=va_iff_phi(u)=phi(v)}. It follows that 
\begin{equation}\label{Eq:I_alpha_+_is_transduction}
I_\alpha^+ = \varphi^{-1}(\IP_{\Gamma_1}^{M_1})\alpha - \{ \alpha \} = \varphi^{-1}(\IP_{\Gamma_1}^{M_1}) \alpha \cap \Sigma_1^+.
\end{equation}
As $L(M)$ has word problem in $\cc$, it follows that $\IP_{\Gamma_1}^{M_1} \in \cc$. As $\cc$ is an $\AFL$ and $\varphi$ is a homomorphism, it hence follows from \eqref{Eq:I_alpha_+_is_transduction} that $I_\alpha^+ \in \cc$. 

Let now $\loz$ be a new symbol. Let $A_\loz = A \cup \{ \loz \}$, and define the homomorphism $\sigma_\loz \colon A_\loz^\ast \to A^\ast$ by $a \mapsto a$ for all $a \in A$, and $\loz \mapsto \alpha$. Define a new rewriting system 
\[
I_\alpha^\loz = \{ (W \to \loz) \mid W \in \sigma_\loz^{-1}(I_\alpha^+) \cap (A_\loz^\ast - A_\loz^\ast \alpha A_\loz^\ast)\}.
\]
Clearly $I_\alpha^\loz$ is a monadic rewriting system. The language of left-hand sides of $\loz$ in $I_\alpha^\loz$ is the intersection $\sigma_\loz^{-1}(I_\alpha^+) \cap (A_\loz^\ast - A_\loz^\ast \alpha A_\loz^\ast)$, and as $I_\alpha^+ \in \cc$ it follows from the closure of $\cc$ under rational transduction that $I_\alpha^\loz$ is a $\cc$-rewriting system. 

Now, as $\alpha$ is self-overlap free, $\Sigma_1$ is a suffix code, so for any word $u \in A^\ast \alpha A^\ast$ we can \textit{uniquely} factor $u$ as $u_0 \alpha u_1 \alpha \cdots \alpha u_k$, where $u_i \in A^\ast - A^\ast \alpha A^\ast$ for $0 \leq i \leq k$. Hence there is exactly one word $u_\loz \in A_\loz^\ast$ with the properties that (1) $u_\loz$ does not contain $\alpha$; and (2) $\sigma_\loz(u_\loz) = u$. The uniqueness of this word depends on the fact that $\alpha$ is self-overlap free.\footnote{For example, if we take the word $\alpha \equiv xyx$, which has self-overlaps, then $\sigma_\loz(xy\loz) = \sigma_\loz(\loz yx) = xyxyx$.} Of course, for any $u \in A^\ast - A^\ast \alpha A^\ast$, there is also a unique such $u_\loz$, namely $u_\loz \equiv u$. Let $L_\loz$ be the language $\{ w_\loz \mid w \in L \}$. Then $\sigma_\loz(L_\loz) = L$. Furthermore, by the above uniqueness argument, we have 
\begin{equation}\label{Eq:sigma_loz_is_easy}
L_\loz = \sigma_\loz^{-1}(L) \cap (A_\loz^\ast - A_\loz^\ast \alpha A_\loz^\ast).
\end{equation}
Hence, from $L \in \cc$ we conclude $L_\loz \in \cc$. 

We now claim that $\langle L \rangle_{I_\alpha} = \sigma_\loz(\langle L_\loz \rangle_{I_\alpha^\loz})$. The right-hand side is the image under the homomorphism $\sigma_\loz$ of the ancestor of $L_\loz$ under the monadic $\cc$-rewriting system $I_\alpha^\loz$. As $\cc$ is a super-$\AFL$, the right-hand side is in $\cc$, and thus we would conclude that $I_\alpha$ is $\cc$-ancestry preserving. The desired equality is easy to prove. Indeed, it follows directly from the fact that if $w \in A^\ast$ and $u \in L$, then $w \xra{I_\alpha} u$ if and only if $w_\loz \xra{I_\alpha^\loz} u_\loz$. This latter fact is proved by an easy induction on the number of rules applied, and we omit this proof. We conclude that $I_\alpha$ is $\cc$-ancestry preserving. 

The case of $I_\alpha^r$ is symmetric, bearing in mind that (1) $\cc$ is reversal-closed; (2) $\alpha^\trev$ is self-overlap free if and only if $\alpha$ is self-overlap free, and (3) $\Sigma_1^\trev$ is a prefix code, rather than a suffix code (factorisation over $\Sigma_1^\trev$ is still unique). We omit the details. 
\end{proof}

One more step is needed before proving Lemma~\ref{Lem:If_L(M)_C_then_aaWP_C}. Let $\cP_2 = \{ w \# w^\trev \mid w \in \Sigma_2^\ast \}$, where the $\cP$ abbreviates ``palindrome''. By \eqref{Eq:SIGMA2ua=va_iff_u=v}, we have 
\begin{equation}\label{Eq:SIGMA2_WP_Palindrome}
\WP_A^M \cap (\Sigma_2^\ast \# (\Sigma_2^\trev)^\ast) = \cP_2.
\end{equation}
Clearly $\cP_2$ is a context-free language, as $\Sigma_2$ is regular, so $\cP_2 \in \cc$. Let $\tau_{\#} \colon A_\#^\ast \to A_\#^\ast$ be defined by $\tau_{\#}(a) = a$ for all $a \in A$, and $\tau_\#(\#) = \alpha \# \alpha^\trev$. We define the (rather complicated-looking) language
\begin{equation}\label{Def:Of_La_langugage}
\cl_\alpha = \big( \tau_\# \big( \WP_{\Sigma_1}^{L(M)} \star\: \cP_2 \big) \big)^{I_\alpha, I_\alpha^r}.
\end{equation}

We shall see in the below proof of Lemma~\ref{Lem:If_L(M)_C_then_aaWP_C} that $\cl_\alpha = \aaWP_A^M$. We first prove that $\cl_\alpha$ is a language encoding the language-theoretic properties of $L(M)$. 

\begin{lemma}\label{Lem:If_L(M)_is_C_then_La_is_C}
If $L(M)$ has word problem in $\cc$, then $\cl_\alpha \in \cc$. 
\end{lemma}
\begin{proof}
If $L(M)$ has word problem in $\cc$, we have $\WP_{\Sigma_1}^{L(M)} \in \cc$ (cf. also the remark preceding \eqref{Eq:Using_piSigma}). This is a concatenation-closed language, as it is a word problem. Furthermore, $\cP_2$ is a context-free language and is easily checked to be concatenation-closed. Hence by Proposition~\ref{Prop:Alternating_preserves}, the alternating product $ \WP_{\Sigma_1}^{L(M)} \star\: \cP_2$ is in $\cc$, as is its image under the homomorphism $\tau_\#$. By Lemma~\ref{Lem:I_a_ancestrypreserving}, $I_\alpha$ and $I_\alpha^r$ are $\cc$-ancestry preserving, so by Proposition~\ref{Prop:Ancestor_preserves}, we have $\cl_\alpha \in \cc$.
\end{proof}

\begin{proof}[Proof of Lemma~\ref{Lem:If_L(M)_C_then_aaWP_C}]
It suffices, by Lemma~\ref{Lem:If_L(M)_is_C_then_La_is_C}, to prove that $\aaWP_A^M = \cl_\alpha$. We prove the inclusions one at a time. 

$(\subseteq)$. Let $w \in \aaWP_A^M$ be arbitrary. Then we can write $w \equiv u\alpha \# (v\alpha)^\trev$ with $u, v \in (\Sigma_1 \cup \Sigma_2)^\ast$ and $u \alpha =_M v \alpha$. Let $u', v'$ be normal forms of $u$ resp. $v$ such that  $u \xra{I_\alpha} u'$ and $v \xra{I_alpha} v'$. Then by Lemma~\ref{Lem:LastM_is_FP} we can write
\[
u' \equiv u_0' u_1' \cdots u_m' \quad \textnormal{and} \quad v' \equiv v_0' v_1' \cdots v_n',
\]
with have $u_i', v_i' \in \Sigma^\ast_{X(i)}$ and $u_i' \alpha =_M v_i'\alpha$ for all $0 \leq i \leq n$, for some standard parametrisation $X$. Let $X$ be such a parametrisation. Let
\begin{equation}\label{Eq:w_is_alternating_prod}
w' \equiv u_1' u_2' \cdots u_n' \# (v_n')^\trev \cdots (v_2')^\trev (v_1')^\trev.
\end{equation}

Now for every $0 \leq i \leq n$, we have $u_i' \alpha =_M v_i' \alpha$ if and only if $\varphi(u_i') =_{M_{X(i)}} \varphi(v_i')$. When $X(i) = 1$, then by \eqref{Eq:Using_piSigma} and \eqref{Eq:SIGMA1ua=va_iff_phi(u)=phi(v)} it follows that $\pi_{\Sigma_1}(u_i') = \pi_{\Sigma_1}(v_i')$, so $u_i' \# (v_i')^{\trev} \in \WP_{\Sigma_1}^{L(M)}$.  When $X(i) = 2$, then by \eqref{Eq:SIGMA2ua=va_iff_u=v} we have $u_i' \equiv v_i'$, so $u_i' \# (v_i')^{\trev} \in \cP_2$. It follows that the right-hand side of \eqref{Eq:w_is_alternating_prod} is an element of the alternating product of $\WP_{\Sigma_1}^{L(M)}$ by $\cP_2$, and hence so too is $w'$. Let $\mathcal{Q} = \WP_{\Sigma_1}^{L(M)} \star \cP_2$, i.e. we just proved $w' \in \mathcal{Q}$. Let
\begin{equation}\label{Eq:w''_is_image_of_w'}
w'' \equiv u_1' u_2' \cdots u_n' \alpha \# \alpha^\trev (v_n')^\trev \cdots (v_2')^\trev (v_1')^\trev \equiv u' \# (v')^\trev.
\end{equation}
Then $w'' \equiv \tau_\#(w')$. We have $u \xra{I_\alpha} u'$, so of course $u\alpha \xra{I_\alpha} u'\alpha$. As $v \xra{I_\alpha} v'$, we have $v^\trev \xra{I_\alpha^r} (v')^\trev$ and so also $\alpha^\trev v^\trev \xra{I_\alpha^r} \alpha^\trev (v')^\trev$. Hence, by the definition of $(I_\alpha, I_\alpha^r)$-ancestors, we have 
\begin{equation}\label{Eq:u'v'_is_ancestor}
u \alpha \# \alpha^\trev v^\trev \in (\tau_\#(\mathcal{Q}))^{I_\alpha, I_\alpha^r}.
\end{equation}
But $u \alpha \# \alpha^\trev v^\trev$ is just $w$; and the right-hand side of \eqref{Eq:u'v'_is_ancestor} is $\cl_\alpha$, so we are done. 

$(\supseteq)$ This direction is straightforward, but somewhat lengthy; it does not use any ideas not present in the proof of $(\subseteq)$ besides the $M$-invariance of $I_\alpha$. The proof is therefore omitted (but is detailed in the author's Ph.D. thesis \cite[Lemma~4.2.13]{Thesis}).
\end{proof}

By combining Lemma~\ref{Lem:aaWP=>WP} and Lemma~\ref{Lem:If_L(M)_C_then_aaWP_C}, we conclude that if $L(M)$ has word problem in $\cc$, then $M$ has word problem in $\cc$. This is the difficult direction of the proof of Theorem~\ref{Thm:MAIN}. It remains to prove the converse:

\begin{lemma}\label{Lem:Converse_of_main}
If $M$ has word problem in $\cc$, then $L(M)$ has word problem in $\cc$.
\end{lemma}
\begin{proof}
Indeed, as before let the homomorphism $\tau_\# \colon A_\#^\ast \to A_\#^\ast$ be defined by $\tau_{\#}(a) = a$ for all $a \in A$, and $\tau_\#(\#) = \alpha \# \alpha^\trev$. Then it is easy to check, using \eqref{Eq:Using_piSigma}, that 
\begin{equation}\label{Eq:tau(WP)=WP_cap}
\tau_\#\left( \WP_{\Sigma_1}^{L(M)}\right)  = \WP_A^M \cap \big(\Sigma_1^\ast \alpha \# \alpha^\trev (\Sigma_1^\trev)^\ast \big).
\end{equation}
It is clear that $\tau_\#$ is injective on languages $L \subseteq A^\ast \alpha \# \alpha^\trev A^\ast$, i.e. that for any such language we have $\tau_\#^{-1} \circ \tau_\#(L) = L$. Hence, applying $\tau_\#^{-1}$ to both sides in \eqref{Eq:tau(WP)=WP_cap}, we find that $\WP_{\Sigma_1}^{L(M)}$ is a rational transduction of $\WP_A^M$. As $\cc$ is closed under rational transductions, it follows that $L(M)$ has word problem in $\cc$.
\end{proof}

This completes the proof of Theorem~\ref{Thm:MAIN}.

\section{Corollaries of Theorem~\ref{Thm:MAIN}}\label{Sec:Context-free}

\noindent In this section, we present some corollaries of Theorem~\ref{Thm:MAIN}. In particular, we will show that it is decidable whether a one-relation monoid containing a non-trivial idempotent has context-free word problem (Theorem~\ref{Thm:Dec_if_OR_has_CF}). This answers a generalisation of a question first asked by Zhang in 1992, which was answered affirmatively by the author in \cite{NybergBrodda2020b}. We begin by applying the result to the rational subset membership problem for monoids. 

\subsection{Rational subset membership problem}

Throughout this section we will fix a weakly compressible monoid $M$, defined by the presentation \eqref{Eq:Main_Monoid}, which is compressible with respect to $\alpha$. As the class of context-free (resp. indexed) languages is a super-$\AFL$ closed under reversal, it of course follows from Theorem~\ref{Thm:MAIN} that: $M$ has context-free (resp. indexed) word problem if and only if $L(M)$ does.

One of the direct consequences for a monoid having context-free word problem is decidability of its \textit{rational subset membership problem}. Recall that for a finitely generated monoid $M$, generated by a finite set $A$ and with associated surjective homomorphism $\pi \colon A^\ast \to M$, this decision problem asks: given a regular language $R \subseteq A^\ast$ and a word $w \in A^\ast$, can one decide whether $\pi(w) \in \pi(R)$? The rational subset membership problem clearly specialises to the submonoid membership problem, the divisibility problems, and the word problem for $M$. It is easy to check that any context-free monoid has decidable rational subset membership problem; indeed, if $\pi(w) \in \pi(R)$ if and only if $w$ is equal to some word in $R$, i.e. if and only if $w \in \WP_A^M / \# R$, where $/$ denotes the right quotient. As $\# R$ is a regular language, and $\WP_A^M$ is a context-free language, the quotient is a context-free language; and membership in context-free languages is well-known to be (uniformly) decidable (cf. also \cite[Theorem~3.5]{NybergBrodda2020b}). We conclude: 

\begin{corollary}\label{Cor:CF_L(M)_has_dec_RSMP}
Let $M$ be a weakly compressible monoid. If $L(M)$ has context-free word problem, then the rational subset membership problem for $M$ is decidable. 
\end{corollary}

We do not know whether the statement of Corollary~\ref{Cor:CF_L(M)_has_dec_RSMP} holds when the rational subset membership problem is replaced by the submonoid membership problem, but conjecture this to be the case.

\begin{example}
Let $n, k, \beta_1, \dots, \beta_n$ be natural numbers such that $n > 1$, $1 \leq k \leq n$, and $\beta_i \geq 1$ for all $1 \leq i \leq n$. Let $\Pi = \Pi(n, k, \{ \beta_i \}_{i=1}^n)$ be the monoid defined by 
\begin{equation*}
\pres{Mon}{a_1, a_2, \dots, a_n}{a_1^{\beta_1} a_2^{\beta_2} \cdots a_n^{\beta_n} = a_k}.
\end{equation*}
Then as the rewriting system with the single rule $(a_1^{\beta_1} a_2^{\beta_2} \cdots a_n^{\beta_n}, a_k)$ is complete, monadic, context-free, and defines $\Pi$, it follows that $\Pi$ has context-free word problem \cite[Corollary~3.8]{Book1982b}. Hence, by Theorem~\ref{Thm:MAIN} any monoid which compresses to $\Pi$ has context-free word problem. Let now $A$ be a new alphabet, let $\alpha \in A^\ast$ be self-overlap free, and let $w_i \in \alpha (A^\ast - A^\ast\alpha A^\ast)$ for $1 \leq i \leq n$ be pairwise distinct words. Let $\tau$ be the homomorphism defined by mapping $a_i \mapsto w_i$ for all $1 \leq i \leq n$. Let $\Pi' = \Pi'(n, k, \{ \beta_i \}_{i=1}^n, \{ w_i \}_{i=1}^n)$ be the monoid defined by 
\begin{equation}\label{Eq:General_example}
\Pi' = \pres{Mon}{A}{\tau(a_1^{\beta_1} a_2^{\beta_2} \cdots a_n^{\beta_n}) \alpha = \tau(a_k) \alpha}.
\end{equation}
Then one readily sees that $L(\Pi') = \Pi$. Hence $\Pi'$ has context-free word problem, for any choice of $n, k, \beta_i$ and $w_i$ ($1 \leq i \leq n$). As a concrete example, if $A = \{ x, y \}$ and $\alpha = xy$, then the monoid $\Pi' = \Pi'(3,2, \{ 2,3,4 \},  \{xy,xyx,xyy\})$ defined by 
\begin{equation}\label{Eq:example_xyxyxyxyxy}
\pres{Mon}{x,y}{(xy)^2(xyx)^3(xyy)^4 xy = xyxxy}
\end{equation}
has context-free word problem, as it compresses to the monoid $\Pi(3, 2, \{2,3,4 \})$ defined by 
\begin{equation*}
\pres{Mon}{a_1, a_2, a_3}{a_1^2 a_2^3 a_3^4 = a_2}.
\end{equation*}
Verifying directly that \eqref{Eq:example_xyxyxyxyxy} has context-free word problem appears to be somewhat tedious. 

We conclude that the monoid \eqref{Eq:example_xyxyxyxyxy} has decidable rational subset membership problem, and more generally so too does any monoid of the form \eqref{Eq:General_example}. 
\end{example}

\subsection{Subspecial one-relation monoids}

We turn to the case when $M$ in \eqref{Eq:Main_Monoid} is defined by a single relation $u=v$. Assume without loss of generality that $|u|\geq |v|$. We say that $M$ is \textit{subspecial} if $u \in vA^\ast \cap A^\ast v$. Any special monoid (i.e. when the defining relation is $u=1$) is obviously subspecial. An element $e \in M$ is called an \textit{idempotent} if $e^2 = e$. Of course, the identity element $1$ is always a (trivial) idempotent. Associated to any idempotent $e$ is a \textit{maximal subgroup}, being the set of elements which are invertible with respect to this idempotent (alternatively, the group $\mathscr{H}$-class of $e$). Lallement \cite{Lallement1974} proved that a one-relation monoid $M$ contains a non-trivial idempotent if and only if (1) $M$ is special, and not right cancellative; (2) $|u|>|v|>0$ and $M$ is subspecial. By using Adian's overlap algorithm (see \cite{NybergBrodda2020b}), it is decidable whether a one-relation special monoid is right cancellative. Hence it is decidable whether a one-relation monoid $M$ contains a non-trivial idempotent. 

As $M$ is subspecial, it is clearly weakly compressible. Furthermore, it is not hard to see that $L(M)$ is also subspecial (see e.g. \cite[Lemma~5.4]{Kobayashi2000}). Thus to any subspecial monoid $M$ we can associate a special monoid $L_s(M)$, obtained by iterating weak compression until we arrive at a special monoid. Hence $M$ has context-free word problem if and only if $L_s(M)$ does, by Theorem~\ref{Thm:MAIN} (of course, the same is also true for any reversal-closed super-$\AFL$).  Using the structural results about maximal subgroups of subspecial monoids by Gray \& Steinberg \cite{Gray2018}, we can prove the following generalisation of the Muller-Schupp theorem.  

\begin{theorem}\label{Thm:Subspecial_has_CF_WP_iff_max}
Let $M$ be a subspecial (one-relation) monoid. Then $M$ has context-free word problem if and only if all of its maximal subgroups are virtually free.
\end{theorem}
\begin{proof}
First, assume $M$ is special. Then, by a result due to Malheiro \cite[Theorem~4.6]{Malheiro2005}, all maximal subgroups of $M$ are isomorphic to the group of units $U(M)$ of $M$. The main theorem of \cite{NybergBrodda2020b} states that $M$ has context-free word problem if and only if $U(M)$ is virtually free. The result thus follows in this case. 

Suppose that $M$ is subspecial, but not special. By \cite[Lemma~5.2]{Gray2019}, the maximal subgroups of $M$ are all isomorphic to the group of units of $L_s(M)$, with the exception of the group of units of $M$, which is trivial (and hence virtually free). Let $G$ be one of the non-trivial maximal subgroups of $M$, so that $G \cong U(L_s(M))$. Now $L_s(M)$ is a finitely generated one-relation special monoid. Hence, by classical results, $G$ is a finitely generated one-relator group. 

Assume now that $M$ has context-free word problem. Then, as $G$ is a finitely generated subsemigroup of $M$, it follows from \cite[Proposition~8(b)]{Hoffmann2012} that $G$ has context-free word problem. Hence, by the Muller-Schupp theorem, $G$ is virtually free, as desired. For the converse, assume $G$ is virtually free. By the easy direction of the Muller-Schupp theorem, $G$ has context-free word problem. Then, by the main theorem of \cite{NybergBrodda2020b}, $L_s(M)$ has context-free word problem. Hence, by the main theorem of the present paper, $M$ has context-free word problem. 
\end{proof}

This gives a complete algebraic characterisation of the subspecial monoids with context-free word problem, extending the Muller-Schupp theorem to this class. 

\begin{corollary}\label{Cor:Subspec_RSMP}
Let $M$ be a subspecial one-relation monoid such that all of its maximal subgroups is virtually free. Then $M$ has decidable rational subset membership problem.
\end{corollary}

For example, one can check that the subspecial monoid 
\begin{equation}\label{Eq:Subspecial_EX}
M_3 = \pres{Mon}{x,y}{xy xyy xy x = x}
\end{equation}
compresses (in a single step) to $L_s(M_3) \cong \pres{Mon}{a,b,c}{abca = 1}$. Now it follows from \cite{Adian1966} that the group of units of $L_s(M_3)$ is isomorphic to $\pres{Gp}{p,q}{pqp=1}$, which is isomorphic to $\mathbb{Z}$ by removing the redundant generator $q$. Hence the non-trivial maximal subgroups of $M_3$ are all infinite cyclic, and in particular virtually free. Hence the monoid $M_3$ defined by \eqref{Eq:Subspecial_EX} has context-free word problem by Theorem~\ref{Thm:Subspecial_has_CF_WP_iff_max}, and it has decidable rational subset membership problem by Corollary~\ref{Cor:Subspec_RSMP}.

We end with a final corollary of Theorem~\ref{Thm:Subspecial_has_CF_WP_iff_max}, the statement of which does not use compression or subspeciality.

\begin{theorem}\label{Thm:Dec_if_OR_has_CF}
Let $M$ be a one-relation monoid containing a non-trivial idempotent. Then it is (uniformly) decidable whether $M$ has context-free word problem. 
\end{theorem}
\begin{proof}
Suppose we are given a one-relation monoid presentation $\pres{Mon}{A}{u=v}$ for $M$. Assume without loss of generality that $|u| \geq |v|$. Then as $M$ contains a non-trivial idempotent, either $M$ is special, or else $M$ is subspecial with $|u| > |v| \geq 0$. In the first case, the defining relation is $u=1$, and by \cite[Theorem~B]{NybergBrodda2020b} we can hence decide (uniformly in $u$) whether this special one-relation monoid has context-free word problem. 

In the latter case, we repeatedly compress the relation $u=v$ until we find 
\begin{equation}\label{Eq:Compressed_subspecial}
L_s(M) = \pres{Mon}{B}{w=1}.
\end{equation}
By Theorem~\ref{Thm:MAIN}, $M$ has context-free word problem if and only if the special one-relation monoid \eqref{Eq:Compressed_subspecial} has context-free word problem; this latter problem is (uniformly) decidable by another application of \cite[Theorem~B]{NybergBrodda2020b}. 
\end{proof}

In 1992, Zhang \cite[Problem~3]{Zhang1992d} asked if it is decidable whether a special one-relation monoid has context-free word problem. This was recently answered affirmatively by the author \cite[Theorem~B]{NybergBrodda2020b}. The above Theorem~\ref{Thm:Dec_if_OR_has_CF} thus extends this answer from \textit{special} to \textit{subspecial} one-relation monoids. 

\section*{Acknowledgements}

The research in this article was carried out while the author was a Ph.D. student at the University of East Anglia. The author wishes to thank his supervisor Dr Robert D Gray for many useful comments and discussions.

{
\bibliography{weak_compression.bib} 
\bibliographystyle{amsalpha}
}
\end{document}